\documentclass[letterpaper,12pt]{article}
\usepackage[margin=1in]{geometry}
\usepackage{enumerate}
\usepackage{enumitem}

\usepackage{amssymb}
\usepackage{amsmath}
\usepackage{amsthm}
\usepackage{mathrsfs}
\usepackage{mathtools}
\usepackage{algorithm}
\usepackage{algorithmicx}
\usepackage{algpseudocode}
\usepackage{bbm}
\usepackage{tikz}
\usepackage{standalone}
\usepackage{import} 
\usepackage{comment}
\usetikzlibrary{calc}
\usepackage{subcaption}
\definecolor{red1}{RGB}{230,25,75}
\definecolor{blue1}{RGB}{0,130,200}
\definecolor{yellow1}{RGB}{255, 225, 25}
\definecolor{green1}{RGB}{60,180,75}
\definecolor{darkgrey1}{RGB}{57, 59, 58}

\usepackage{cleveref}  
\Crefname{graph}{Graph}{Graphs}

\newtheorem{theorem}{Theorem}[section]
\newtheorem{proposition}[theorem]{Proposition}

\newtheorem{lemma}[theorem]{Lemma}
\newtheorem{claim}[theorem]{Claim}

\newtheorem{observation}[theorem]{Observation}
\newtheorem{definition}[theorem]{Definition}
\newtheorem{open}[theorem]{Open Problem}

\usepackage[isbn=false,doi=false,style=numeric,maxbibnames=50]{biblatex}

\begin{document}
\title{Metric Approximations of Consistent Path Systems}

\author{Daniel Cizma\thanks{Einstein Institute of Mathematics, Hebrew University, Jerusalem 91904, Israel. e-mail: daniel.cizma@mail.huji.ac.il.} \and  {Nati Linial\thanks{School of Computer Science and Engineering, Hebrew University, Jerusalem 91904, Israel. e-mail: nati@cs.huji.ac.il.{~Supported in part by an ERC Grant 101141253, "Packing in Discrete Domains - Geometry and Analysis".}}}}
\date{}
\maketitle
\begin{abstract}
A {\em path system} $\mathcal{P}$ in a graph $G=(V,E)$ is a collection of paths, 
with exactly one path between any two vertices in $V$. 
A path system is said to be {\em consistent}
if it is closed under subpaths. We say that a path system $\mathcal{P}$ is 
$\alpha$-metric if there exists a metric $\rho$ on $V$ such that
$\sum_{i=1}^{k}\rho(x_{i-1},x_{i}) \le \alpha \rho(x_0,x_k)$
for every path $(x_0,x_1,\dots,x_k)\in \mathcal{P}$.
Also, we denote by $\Delta(\mathcal{P})$ the infimum of $\alpha$
for which $\mathcal{P}$ is $\alpha$-metric. We show that $\Delta(\mathcal{P}) \le O(\sqrt{n})$ for 
every $n$-point consistent path system $\mathcal{P}$. 
On the other hand, we construct infinitely many $n$-point consistent path systems
$\mathcal{P}_n$ with $\Delta(\mathcal{P}_n) \ge \tilde{\Omega}(\sqrt{n})$, 
showing these bounds are tight up to a polylogarithmic factor.
We also show how to efficiently compute $\Delta(\mathcal{P})$
for a given path system.
\end{abstract}
\section{Introduction}
Suppose we need to navigate a graph $G=(V,E)$. Specifically we wish to select 
an `optimal' route between any two given vertices. 
To this end, we choose a {\em path system}, i.e.,
a collection of paths in $G$ specifying a unique connecting path $P_{xy}=P_{yx}$ between 
each pair of vertices
$x,y$. A path system is said to be {\em consistent} if it is closed under taking subpaths. 
The most obvious examples of consistent path systems are obtained by 
assigning positive weights to the edges, 
$w:E\to \mathbb{R}_{>0}$, and letting $P_{xy}$ be the
unique $w$-shortest $xy$ path. Such a consistent path system is said to be
{\em strictly metric}. 
There is already a sizable body of work concerning these objects. 
The present work contributes to an ongoing research program that seeks to 
understand the properties of consistent path system, see \cite{BDG,BBW}. 
Bodwin \cite{Bo}
has characterized strictly metric path systems in topological terms, and as we
showed, strictly
metric path systems are way fewer than the non-metric variety, i.e.,
strictly metric consistent path systems
form a vanishingly small subclass of consistent path systems.
We also proved this by approximately
counting order-$n$ path systems of both varieties \cite{CL3}. 

Our purpose here is to develop a different
perspective of this comparison. Such comparisons are
abundant in mathematics. Thus, while the rationals $\mathbb{Q}$
constitute a dense subset of the reals $\mathbb{R}$, the theory of Diophantine
Approximations seeks to make this comparison quantitative and discover how well
real numbers can be approximated by rationals.
More specifically, our
notion of approximation is inspired by the standard measure of proximity
in the study of metric embeddings, see e.g., \cite{Bour,LLR}.
A consistent path system $\mathcal{P}$ on vertex set 
$X$ is metric iff there is a metric 
$\rho$ on $X$ such that for every path $(x_0,x_1,\dots,x_{k})\in \mathcal{P}$
there holds $\sum_{i=1}^{k}\rho(x_{i-1},x_{i}) = \rho(x_0,x_k)$.

Let $(X,\rho)$ be a finite metric space, and let $(x_0,x_1,\dots,x_{k})$
be a sequence of points in $X$. By the triangle inequality
\[\sum_{i=1}^{k}\rho(x_{i-1},x_{i}) \ge \rho(x_0,x_k).\]
Next let $\mathcal{P}$ be a consistent path system on vertex set $X$.
We say that $\mathcal{P}$ is $\alpha$-metric w.r.t.\ the metric
$\rho$, for some $\alpha \ge 1$
if every path $(x_0,x_1,\dots,x_{k})\in \mathcal{P}$ satisfies
\[\sum_{i=1}^{k}\rho(x_{i-1},x_{i}) \le \alpha\cdot \rho(x_0,x_k).\] 

Finally we let $\Delta(\mathcal{P})$ be the smallest $\alpha \ge 1$ such that 
$\mathcal{P}$ is $\alpha$-metric w.r.t.\ {\em some} metric on $X$.
We ask, as a function of $n$, how large $\Delta(\mathcal{P})$ can be for 
a consistent path system on $n$ points. In \cref{sec:upper} we prove that every consistent $n$-vertex path system satisfies
\[\Delta(\mathcal{P})\le O(\sqrt{n}).\]
In \cref{sec:lower} we show this upper bound is essentially tight by constructing $n$-vertex path systems $\mathcal{P}=\mathcal{P}_n$ for which 
\[\Delta(\mathcal{P}_n)\ge \tilde{\Omega}(\sqrt{n}).\]
Finally, in \cref{sec:realizable} we discuss the possible values that $\Delta(\mathcal{P})$ can take on, showing that $\Delta(\mathcal{P})$ is always algebraic and possibly irrational. Moreover, we show that the minimal polynomial of $\Delta(\mathcal{P})$ can be computed in time polynomial in $n$.

\section{Definitions and Background}\label{sec:background}
Let $G=(V,E)$ be a connected graph. A consistent path system $\mathcal{P}$ in $G$ is a collection of paths in $G$ such that 
\begin{enumerate}
	\item For every $u,v\in V$ there is unique $(uv)$-path $P_{u,v}\in \mathcal{P}$ with $P_{u,v} = P_{v,u}$.
	\item If $P\in \mathcal{P}$ and $x,y\in P$ then the $(xy)$-subpath of $P$ is $P_{x,y}\in \mathcal{P}$.
\end{enumerate}

We say $\mathcal{P}$ is {\em metric} if there exists a weight function 
$w:E\to \mathbb{R}_{>0}$ such that every path in $\mathcal{P}$ is a 
$w$-shortest path. 

As observed in \cite{CL3}, the notion of metric path systems
can also be captured in terms of collinearity. Let
$P=(x_0,x_1,\dots,x_k)$ be a path in a metric space $(V,d)$,
that is, a sequence of points from $V$,
and define $d(P) = \sum_{i=1}^{k} d(x_{i-1},x_{i})$.
We say that $P$ is $d$-collinear if $d(P) = d(x_{1},x_{k})$.
A path system $\mathcal{P}$ on $V$ 
is metric if there exists a metric $\rho$ on $V$
such that the paths in $\mathcal{P}$ are $\rho$-collinear, i.e.,
they satisfy $\rho(P_{u,v}) = \rho(u,v)$ for all $u,v\in V$. 
This definition naturally suggests an approximate notion of metric path systems: 
\begin{definition}
A consistent path system $\mathcal{P}$ on $V$ is 
called $\alpha$-metric if there exists a metric 
$(V,\rho)$ such that  
\[\rho(u,v)\le \rho(P_{u,v}) \le \alpha\cdot \rho(u,v)\]
for every $P_{u,v}\in \mathcal{P}$.
\end{definition}
In other words, a path system is $\alpha$-metric iff every path in 
the system is `$\rho$-sub-additive' 
after scaling distances by a factor of $\alpha\ge 1$. Note
that a path system is metric if and only if it is $1$-metric.\\
Equivalently, $\mathcal{P}$ is $\alpha$-metric if there exists weights 
$w:E\to \mathbb{R}_{>0}$ such that the inequality 
$$\sum_{i=1}^{k}w(x_{i-1}x_i) \le \alpha\cdot \sum_{i=1}^{m}w(y_{i-1}y_i)$$ holds
for every path $P_{u,v} = (x_0,\dots,x_k)$ in $\mathcal{P}$ and any other 
$(uv)$-path $Q=(y_0,\dots, y_m)$.

For a given path system we seek the smallest $\alpha\ge 1$ for which $\mathcal{P}$ 
is $\alpha$-metric. These considerations lead to the following definition:
\begin{definition}
	For a consistent path system $\mathcal{P}$ we define $\Delta(\mathcal{P})$ to be
	\[\Delta(\mathcal{P}) = \inf\{\alpha \ge1 : \mathcal{P} \text{ is $\alpha$-metric}\}.\]
\end{definition}
We start with the trivial bounds $1\le \Delta(\mathcal{P}) \le n-1$. 
To understand this upper bound, consider what happens if we assign unit weights to each edge.
The claim follows, since each path in the system has at most $n-1$ edges, whereas all
pairwise distances are at least $1$. Consequently, 
$\rho(P_{u,v}) = |P_{u,v}| \le n-1 \le (n-1)\cdot \rho(u,v)$. In fact, this argument shows that $\Delta(\mathcal{P}) \le \text{diam}(\mathcal{P})$, where $\text{diam}(\mathcal{P}) = \max_{P\in \mathcal{P}} |P|$.
\begin{figure}[ht]
	\centering
	\begin{subfigure}[t]{0.4\textwidth}
		\centering
		\begin{tikzpicture}[scale=0.5, every node/.style={scale=0.45}]

		\node[draw,circle,minimum size=.5cm,inner sep=.1pt, scale =1.3] (1) at (-0*360/5 +90: 5cm) [scale=1.6] {$1$};
		\node[draw,circle,minimum size=.5cm,inner sep=.1pt, scale =1.3] (2) at (-1*360/5 +90: 5cm) [scale=1.6]{$2$};
		\node[draw,circle,minimum size=.5cm,inner sep=.1pt, scale =1.3] (3) at (-2*360/5 +90: 5cm)[scale=1.6] {$3$};
		\node[draw,circle,minimum size=.5cm,inner sep=.1pt, scale =1.3] (4) at (-3*360/5 +90: 5cm) [scale=1.6]{$4$};
		\node[draw,circle,minimum size=.5cm,inner sep=.1pt, scale =1.3] (5) at (-4*360/5 +90: 5cm)[scale=1.6] {$5$};
		
		\node[draw,circle,minimum size=.5cm,inner sep=.1pt, scale =1.3] (6) at (-5*360/5 +90: 2.5cm) [scale=1.6]{$6$};
		\node[draw,circle,minimum size=.5cm,inner sep=.1pt, scale =1.3] (7) at (-6*360/5 +90: 2.5cm) [scale=1.6]{$7$};
		\node[draw,circle,minimum size=.5cm,inner sep=.1pt, scale =1.3] (8) at (-7*360/5 +90: 2.5cm) [scale=1.6]{$8$};
		\node[draw,circle,minimum size=.5cm,inner sep=.1pt, scale =1.3] (9) at (-8*360/5 +90: 2.5cm) [scale=1.6]{$9$};
		\node[draw,circle,minimum size=.5cm,inner sep=.1pt, scale =1.3] (10) at (-9*360/5 +90: 2.5cm)[scale=1.6] {$10$};
		
		\coordinate (c1) at ($(7)!0.5!(10)$);
		\coordinate (c2) at ($(6)!0.5!(8)$);
		\coordinate (c3) at ($(7)!0.5!(9)$);
		\coordinate (c4) at ($(8)!0.5!(10)$);
		\coordinate (c5) at ($(6)!0.5!(9)$);
		
		\draw [line width=3.1pt,-,red1] (1) -- (2);
		\draw [line width=3.1pt,-,blue1] (1) -- (5);
		\draw [line width=3.1pt,-,red1] (1) -- (6);
		\draw [line width=3.1pt,-,green1] (2) -- (3);
		\draw [line width=3.1pt,-,green1] (2) -- (7);
		\draw [line width=3.1pt,-,darkgrey1] (3) -- (4);
		\draw [line width=3.1pt,-,darkgrey1] (3) -- (8);
		\draw [line width=3.1pt,-,yellow1] (4) -- (5);
		\draw [line width=3.1pt,-,yellow1] (4) -- (9);
		\draw [line width=3.1pt,-,blue1] (5) -- (10);
		\draw [line width=3.1pt,-,red1] (6) -- (8);
		\draw [line width=3.1pt,-,yellow1] (6) -- (9);
		\draw [line width=3.1pt,-,green1] (7) -- (9);
		\draw [line width=3.1pt,-,blue1] (7) -- (10);
		\draw [line width=3.1pt,-,darkgrey1] (8) -- (10);
		
		\draw [line width=3.1pt,-,red1] (c2) -- (6);
		\draw [line width=3.1pt,-,yellow1] (c5) -- (9);

		\end{tikzpicture}
		\caption{A path system in the Petersen Graph}
		\label{fig:nonMetPetersen}
	\end{subfigure}
	\hspace{20mm}
	\begin{subfigure}[t]{0.4\textwidth}
		\centering
		\begin{tikzpicture}[scale=0.5, every node/.style={scale=0.45}]

		\node[draw,circle,minimum size=.5cm,inner sep=.1pt, scale =1.3] (1) at (-0*360/5 +90: 5cm) [scale=1.6] {$1$};
		\node[draw,circle,minimum size=.5cm,inner sep=.1pt, scale =1.3] (2) at (-1*360/5 +90: 5cm) [scale=1.6]{$2$};
		\node[draw,circle,minimum size=.5cm,inner sep=.1pt, scale =1.3] (3) at (-2*360/5 +90: 5cm)[scale=1.6] {$3$};
		\node[draw,circle,minimum size=.5cm,inner sep=.1pt, scale =1.3] (4) at (-3*360/5 +90: 5cm) [scale=1.6]{$4$};
		\node[draw,circle,minimum size=.5cm,inner sep=.1pt, scale =1.3] (5) at (-4*360/5 +90: 5cm)[scale=1.6] {$5$};
		
		\node[draw,circle,minimum size=.5cm,inner sep=.1pt, scale =1.3] (6) at (-5*360/5 +90: 2.5cm) [scale=1.6]{$6$};
		\node[draw,circle,minimum size=.5cm,inner sep=.1pt, scale =1.3] (7) at (-6*360/5 +90: 2.5cm) [scale=1.6]{$7$};
		\node[draw,circle,minimum size=.5cm,inner sep=.1pt, scale =1.3] (8) at (-7*360/5 +90: 2.5cm) [scale=1.6]{$8$};
		\node[draw,circle,minimum size=.5cm,inner sep=.1pt, scale =1.3] (9) at (-8*360/5 +90: 2.5cm) [scale=1.6]{$9$};
		\node[draw,circle,minimum size=.5cm,inner sep=.1pt, scale =1.3] (10) at (-9*360/5 +90: 2.5cm)[scale=1.6] {$10$};

		\coordinate (c1) at ($(7)!0.5!(10)$);
		\coordinate (c2) at ($(6)!0.5!(8)$);
		\coordinate (c3) at ($(7)!0.5!(9)$);
		\coordinate (c4) at ($(8)!0.5!(10)$);
		\coordinate (c5) at ($(6)!0.5!(9)$);
		
		\draw [line width=3.1pt,-,red1] (1) -- node[yshift = 2mm,xshift = 1mm, above, color=black, scale= 1.75] {$1$} ++(2) ;
		\draw [line width=3.1pt,-,red1] (1) -- (5);
		\draw [line width=3.1pt,-,red1] (1) -- (6);
		\draw [line width=3.1pt,-,red1] (2) -- (3);
		\draw [line width=3.1pt,-,red1] (2) -- (7);
		\draw [line width=3.1pt,-,red1] (3) -- (4);
		\draw [line width=3.1pt,-,red1] (3) -- (8);
		\draw [line width=3.1pt,-,red1] (4) -- (5);
		\draw [line width=3.1pt,-,red1] (4) -- (9);
		\draw [line width=3.1pt,-,red1] (5) -- (10);
		\draw [line width=3.1pt,-,blue1] (6) -- (8);
		\draw [line width=3.1pt,-,blue1] (6) -- (9);
		\draw [line width=3.1pt,-,blue1] (7) -- (9);
		\draw [line width=3.1pt,-,blue1] (7) -- node[below, color=black, scale= 1.75] {$\varepsilon$} ++(10);
		\draw [line width=3.1pt,-,blue1] (8) -- (10);
		
		\draw [line width=3.1pt,-,blue1] (c2) -- (6);
		\draw [line width=3.1pt,-,blue1] (c5) -- (9);

		\end{tikzpicture}
		\caption{A weighted Petersen graph where red edges have weight $1$ and blue edges have weight $\varepsilon$}
		\label{fig:PetersenWeigths}
	\end{subfigure}
	\caption{$\mathcal{P}$ is a consistent path system in the Petersen graph. It is non-metric, and yet $\Delta(\mathcal{P})=1$.}
\end{figure}
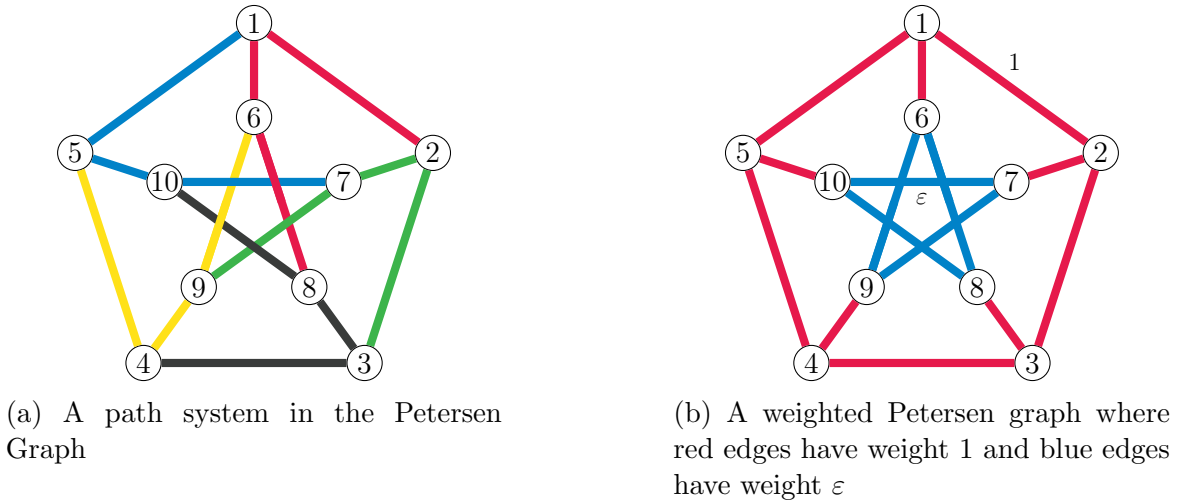

We remark that if $\mathcal{P}$ is metric then clearly $\Delta(\mathcal{P})=1$. 
However, the converse is not true as the following example shows. Consider
the following path system $\mathcal{P}$ on the Petersen graph $\Pi$,
that has three types of paths:
\begin{enumerate}
	\item For every pair of adjacent vertices $u,v$, the path $P_{u,v}$ is the edge $uv$
	\item The five colored $3$-hop paths in \cref{fig:nonMetPetersen} belong to $\mathcal{P}$
	\item For every other pair of vertices $x,y$, $P_{x,y}$ is the unique $2$-hop path connecting $x$ and $y$.  
\end{enumerate}
It is not difficult to show (see \cite{CL}) that $\mathcal{P}$ is consistent and not 
metric. Therefore, if $\mathcal{P}$ is $\alpha$-metric, then $\alpha >1$, but
in fact, $\mathcal{P}$ is $\alpha$-metric for every $\alpha >1$. 
Fix some $\varepsilon>0$ and let $\rho$ be the shortest path metric obtained
when in \cref{fig:PetersenWeigths} red edges have a weight of $1$
and blue edges a weight of $\varepsilon$. It is easily verified that
$\mathcal{P}$ is $(1+\frac{\varepsilon}{2})$-metric
with respect to this metric. 

Since $\Delta(\mathcal{P})$ measures how far $\mathcal{P}$ is from
being metric, the question that we are pursuing asks how large 
$\Delta(\mathcal{P})$ can be for an $n$-point space.
\subsection{Group Invariant Path Systems}\label{sec:gips}
Like many other combinatorial objects, it is interesting to
investigate path systems with certain {\em symmetries}.
Thus, we managed to construct in \cite{CL2} irreducible path systems
(for the definition of irreducibility, consult that paper)
on Paley graphs. That argument
depends crucially on particular symmetries of these path systems.
Here, we extend this idea and introduce {\em group invariant} path systems. 
{\it The theory developed here is used subsequently (\cref{sec:lower}) 
to construct consistent path systems of large $\Delta(\mathcal{P})$.} 
Let $\Gamma$ be a Cayley graph of a group $G$. For a simple
path $P=(h_1,\dots,h_k)$ in $\Gamma$ and $g\in G$ we define 
\[gP \coloneqq (gh_1,\dots,gh_k).\]
\begin{definition}
A path system $\mathcal{P}$ in a Cayley graph of a group $G$ is
said to be {\em $G$-invariant} if for all $g,x,y\in G$ there holds
    \[gP_{x,y} = P_{gx,gy}\]
\end{definition}
Let $\Gamma$ be a Cayley graph of a group $G$. Let us elucidate the notion of 
$G$-invariance for a path system on $\Gamma$. Consider the mapping $G\mapsto G$
where $x\mapsto gx$ for some given $g\in G$. 
This is clearly an automorphism of $\Gamma$. A path system 
$\mathcal{P}$ in $\Gamma$ is $G$-invariant iff the same maps induce
an automorphism of $\mathcal{P}$ as well. 
Consequently, a $G$-invariant path system $\mathcal{P}$ in a Cayley graph of $G$ is
uniquely determined by the paths $P_{e,x}$, where $e$ is $G$'s identity element
and $x\neq e$ ranges over all other elements of $G$. Say that $P_{e,x}= 
(e, g_1, g_1\cdot g_2, g_1\cdot g_2\cdot g_3, \ldots, g_1 \cdot g_2\ldots \cdot g_r)$,
where $g_1, g_2,\ldots, g_r$ are generators of the Cayley graph, and
$x=g_1 \cdot g_2\ldots \cdot g_r$. We denote the word $g_1, g_2, \ldots, g_r$ by $w_x$.\\
%Let us further assume that $\mathcal{P}$ is also consistent.
%If $g_i\cdot\ldots\cdot g_j=y$
%for some $1\le i\le j\le r$, then necessarily $w_y=g_i, g_{i+1}, \ldots, g_j$. In fact:
As we show next, a $G$-invariant path system is conveniently defined by 
a collection of words that satisfy certain consistency conditions.
\begin{proposition}\label[proposition]{prop:words}
Let $G$ be a group and $S^{-1}=S\subset G$ a generating set. 
For every $x\in G$ pick a presentation
$x = \gamma_1\cdot \gamma_2\cdots \gamma_t$ where 
$\gamma_1,\dots, \gamma_t\in S$, and let $w_x$
be the word $w_x = \gamma_{1},\dots, \gamma_t$.
Suppose this collection of words $\{w_x:x\in G\}$ satisfies: 
\begin{enumerate}
\item If $w_x= \gamma_{1},\dots, \gamma_t$ then 
$w_{x^{-1}}=\gamma_t^{-1},\dots, \gamma_1^{-1}$
\item If $w_x= \gamma_{1},\dots, \gamma_t$, and if 
$y=\gamma_{i}\cdots \gamma_j$
for some $1\le i \le j \le t$, then  $w_{y}=\gamma_{i},\dots, \gamma_j$.
\end{enumerate}
Then the words $\{w_x:x\in G\}$ uniquely define a $G$-invariant consistent path system by setting 
\[P_{e,x} = 
(e,\gamma_1,\gamma_1\cdot \gamma_2,\dots, \gamma_1\cdot \gamma_2\cdots \gamma_t),
\]
where $w_x = \gamma_{1},\dots, \gamma_t$. 
\end{proposition}
\begin{proof}
We define $\mathcal{P}$ as follows: for all $x,y\in G$ we set 
\[P_{x,y} \coloneqq x P_{e, x^{-1}y}.\]
Let us verify that with this definition $P_{x,y} = P_{y,x}$.
If $w_{x^{-1}y} = g_1,\dots, g_r$ then 
\[P_{x,y} = (x,x\cdot g_1,x\cdot g_1\cdot g_2,\dots, x\cdot g_1\cdot g_2\cdots g_r=y).\]
Note that property $1$ yields $w_{y^{-1}x} = g_r^{-1},\dots, g_1^{-1}$ and
\[P_{y,x} = (y,y\cdot g_r^{-1},y\cdot g_r^{-1}\cdot g_{r-1}^{-1},\dots, y\cdot g_r^{-1}\cdot g_{r-1}^{-1}\cdots g_1^{-1}=x).\]
Indeed, these two paths coincide. 
Notice that the same edge is labeled $g$ resp.\ $g^{-1}$
according to the orientation at which it is traversed. \\
Next we show that this system is consistent. We need to show that if
$P_{x,y}$ and $P_{u,v}$ have two common vertices, say $a$ and $b$, 
then the $(ab)$-subpaths of these two paths coincide. 
Let us write $w_{x^{-1}y}=g_1,\dots, g_r$ and $w_{u^{-1}v}=h_1,\dots, h_s$ and 
\[P_{x,y}= (x,xg_1,xg_1\cdot g_2,\dots, xg_1\cdot g_2\cdots g_r) \hspace{15mm}P_{u,v}= (u,uh_1,uh_1\cdot h_2,\dots, uh_1\cdot h_2\cdots h_s).\]
We may assume $a=xg_1\cdots g_i = uh_1\cdots h_{i'}$ and  
$b=xg_1\cdots g_j = uh_1\cdots h_{j'}$ with $i<j$ and $i'<j'$. 
(If, say, $j<i$, we work instead with $w_{y^{-1}x}$ .) Observe that
\[g_{i+1}\cdots g_{j} = a^{-1}b = h_{i'+1}\cdots h_{j'}.\]
Since $g_{i+1},\dots, g_{j}$ and $ h_{i'+1},\dots, h_{j'}$ are subwords of 
$w_{x^{-1}y}$ and $w_{u^{-1}v}$ by property 2 
\[\text{the words~}g_{i+1},\dots, g_{j} \text{~and~} h_{i'+1},\dots, h_{j'}\text{~coincide},\]
and hence these two $(ab)$-subpaths coincide. \\
That this system is $G$-invariant follows from the fact that $(x,y)=(gu,gv)$ for some $g\in G$ iff $x^{-1}y=u^{-1}v$. In this case, 
    \[P_{x,y} =xP_{e,x^{-1}y}=gu P_{e,u^{-1}v}= gP_{u,v}.\]
\end{proof}

A metric $\rho$ on the elements of a group $G$
is called $G$-invariant if $\rho(gu,gv) = \rho(u,v)$ for all $g,u,v\in G$.
The following proposition will be used in the sequel.

\begin{proposition}\label[proposition]{prop:invariant_metric}
Let $\mathcal{P}$ be an $\alpha$-metric $G$-invariant consistent path system.
Then there exists a $G$-invariant metric $\rho$ such that
$\rho(P_{u,v}) \le \alpha \rho(u,v)$ for all $u,v\in G$.
\end{proposition}
\begin{proof}
Since $\mathcal{P}$ is $\alpha$-metric, 
there exists some metric $\tilde{\rho}$ such 
that $\tilde{\rho}(P_{u,v}) \le \alpha \tilde{\rho}(u,v)$ for all $u,v\in G$.    
It is a common practice 
with such problems to act on this inequality with some $g\in G$,
and sum over $g$. Thus,
we define the metric $\rho$ on $G$ via 
\[\rho(u,v) \coloneqq \sum_{g\in G} \tilde{\rho}(gu,gv).\]
This metric $\rho$ is clearly $G$-invariant. 
Moreover, for $P_{u,v} = (u=x_0,\dots, x_{r}=v)$ there holds  
\[\rho(P_{u,v}) = \sum_{i=1}^{r}\rho(x_{i-1},x_{i}) = \sum_{i=1}^{r}\sum_{g\in G}\tilde{\rho}(gx_{i-1},gx_{i}) =\sum_{g\in G}\sum_{i=1}^{r}\tilde{\rho}(gx_{i-1},gx_{i}) = \sum_{g\in G} \tilde{\rho}(P_{gu,gv}),\]
since $\mathcal{P}$ is $G$-invariant.
In particular, this implies that for all $u,v\in G$
\begin{equation*}
\rho(P_{u,v}) =\sum_{g\in G} \tilde{\rho}(P_{gu,gv}) \le \sum_{g\in G} 
\alpha\cdot \tilde{\rho}(gu,gv) = \alpha\cdot \rho(u,v).
\end{equation*}
\end{proof}

\section{Upper Bounds}\label{sec:upper}

In this section we prove an upper bound for $\Delta(\mathcal{P})$. Before stating this result we briefly discuss the trivial upper bound  mentioned in \cref{sec:background}:
every $n$-vertex consistent path system $\mathcal{P}$ satisfies $\Delta(\mathcal{P}) \le n-1$. 
Actually, this simple argument
applies as well to {\em all} path systems, consistent or not, and as
the following construction shows, in this generality the bound is tight.
Let $C_n$ be the $n$-vertex cycle $(x_1,x_2\cdots x_{n},x_{1})$. For every pair of adjacent vertices $x_{i},x_{i+1}$ we set $P_{x_{i}, x_{i+1}}$ to be the ``long'' $(x_{i}x_{i+1})$-path in $C_n$, i.e. the path in $C_n$ obtained by deleting the edge $x_{i}x_{i+1}$. For every other pair of vertices we may choose the path arbitrarily. Clearly, this system is inconsistent. Now suppose that for some $\alpha\ge 1$ and metric $\rho$ we have 
$\rho(P_{x_{i},x_{i+1}}) \le \alpha \rho(x_{i},x_{i+1})$ for all $1\le i \le n$. If $i_0$ is an
index such that  $\rho(x_{i_0},x_{i_0+1})$ is smallest, then 
\[(n-1) \rho(x_{i_0},x_{i_0+1}) \le \rho(P_{x_{i_0},x_{i_0+1}}) \le \alpha \rho(x_{i_0},x_{i_0+1}),\]
implying $n-1\le \alpha$. \\
The use of inconsistency was crucial here, since, as shown in \cite{CL}, every
consistent path systems in $C_n$ is strictly metric. So to improve this upper bound, we must explicitly use 
the consistency of the path system at hand.

With this in mind we have the following upper bound:
\begin{theorem}\label[theorem]{thm:upperbound}
Every consistent path system $\mathcal{P}$ on $n$ vertices satisfies
\[\Delta(\mathcal{P}) \le 2\sqrt n.\]
\end{theorem}

Our proof actually proves much more. We find a system
of edge weights for which every path
in $\mathcal{P}$ weighs at most
$2\sqrt{n}$, while every path in $G$ that is
not in $\mathcal{P}$ weighs at least $1$.

\begin{lemma}\label[lemma]{lem:basic}
For every consistent $n$-vertex path system 
$\mathcal{P}$ in a graph $G=(V,E)$, 
there exists a weight function $w:E\to \mathbb{R}_{>0}$ for which
\begin{enumerate}[label=(\roman*)]
\item $w(P) \le 2\sqrt n$ for all $P\in \mathcal{P}$
\item $w(P) \ge 1$ for every path
$P$ in $G$ with $P\notin \mathcal{P}$
\end{enumerate}
\end{lemma}

Clearly, lemma \ref{lem:basic} implies \cref{thm:upperbound}. 
In our proof, the weight function $w$ takes only two values: 
$\frac{1}{n}$ and $1$. Let $F$
be the graph whose edge set is precisely the edges on which $w$ equals
$\frac{1}{n}$, and let $P$ be a path contained in $F$.
Clearly, $w(P) < 1$, so by condition (ii) of the lemma,
$P$ must belong to $\mathcal{P}$. But if every path that is contained in $F$
belongs to $\mathcal{P}$, since $\mathcal{P}$ 
is consistent, the edge set $F$ is acyclic, i.e.,
the edge-set of a sub-forest of $G$. Moreover, by condition (i) every path in $\mathcal{P}$ can have at most $2\sqrt n$ edges not in $F$. It is not difficult to see that these conditions are not only necessary but also sufficient for $w:E\to \{\frac{1}{n},1\}$ to satisfy the requirements of \cref{lem:basic}.
What we need then is the following lemma. Recall a forest is said be {\em linear} if every connected component is a path
\begin{lemma}\label[lemma]{lem:forest}
If $\mathcal{P}$ is a consistent path system in an $n$-vertex graph $G=(V,E)$,
then $G$ has a linear sub-forest $F$ for which
\begin{enumerate}[label=(\roman*)]
\item Every path included in $F$ is a member of $\mathcal{P}$.
\item Every path in $\mathcal{P}$ has fewer than $2\sqrt n$ edges outside of $F$.
\end{enumerate}
\end{lemma}
We remark that the fact that $F$ is a linear forest rather than just forest is not essential for the proof of \cref{lem:basic}, 
but we feel that this fact is interesting in its own right. We can rephrase \cref{lem:forest} as follows: for any consistent path system $\mathcal{P}$ we may find a subset $S\subset \mathcal{P}$ of vertex disjoint subpaths such that the paths in $\mathcal{P}$ are ``mostly'' covered by the paths in $S$, with at most $2\sqrt n$ edges not covered on each path.
We now prove lemma \ref{lem:forest} and with that 
establish \cref{thm:upperbound}.
\begin{proof}
We set $D\coloneqq 2\sqrt{n}$. 
We construct $F$ sequentially, starting with $F_0$, an empty set of edges. 
At each stage $i$ we 
consider a path $P_i\in \mathcal{P}$ such that 
$|E(P_i)\setminus E(F_{i-1})| \ge D$. 
If no such path exists then the process terminates. Otherwise, let
\begin{equation}\label{eq:what_remains}
Q_i:=P_i \setminus V(P_i\cap F_i). 
\end{equation}
Note that $Q_i$ is obtained from $P_i$ by removing
from $P_i$ the {\em vertices, not the edges,} that it shares with $F_i$. 
Clearly, $Q_i$ comprises a collection of vertex disjoint paths 
that are also disjoint from $F_i$,
and we let $F_{i+1}:=F_i\cup E(Q_i)$. By construction, 
and by consistency, $F_{i+1}$ is still the disjoint union of paths in 
$\mathcal{P}$. We further observe that

\begin{equation}\label{eq:growth}
|F_{i+1}|\ge |F_i|+D-2i.
\end{equation}

To see this, consider the intersection of two paths $P, P'$ in a consistent path system.
This is a path that is a (possibly empty) subpath of both $P$ and $P'$.
In our case, the term $V(P_i\cap F_i)$ in equation (\ref{eq:what_remains}), records the
intersection of $P_i$ with various $P_j$ with $j<i$. In particular, an edge $e\in P_i$ can only fail to appear in $F_{i}$ if $e$ is incident to, but not contained in, some path $P_{i}\cap P_{j}$, with $j<i$. 
Therefore, every $j<i$ creates
a loss of at most two edges of $P_i$, namely the one 
immediately prior and immediately
following the common subpath of $P_i$ and $P_j$. 

It remains to show that this process
terminates. Indeed, we show that it
terminates within $D/2$ steps.
There holds $|F_{\nu}|\le n-1$ for every 
index $\nu$, since $F_{\nu}$ is a forest.
Summing equation (\ref{eq:growth})
over $i\le D/2$, yields:

\[D^2/4\le n-1,\]

\noindent
which is a contradiction, since we took $D=2\sqrt n$.

\end{proof}

\section{Lower Bounds}\label{sec:lower}

In this section we show that the upper bound from
the previous section is tight, up to polylogarithmic factors. In particular, we prove the following theorem:
\begin{theorem}\label[theorem]{thm:main}
There exist arbitrarily large order-$n$ groups $G_n$ along with
corresponding Cayley graphs $\Gamma_n$ and consistent $\Gamma_n$-invariant
path systems $\mathcal{P}_n$ 
with $\Delta(\mathcal{P}_n) \ge \tilde{\Omega}(\sqrt{n})$.
\end{theorem}

We begin with a lemma that provides a general framework 
that yields consistent group-invariant path systems with large $\Delta$.
% Let $G$ be a finite group, $g\in G$ a groups element
% and $m$ a non-negative integer. We define $D_m(g) \coloneqq \{g^i : -m\le i \le m\}$. Then:
\begin{lemma}\label[lemma]{lem:cayley_metric}
Let $G$ be a finite group, and $X=X^{-1} \subset G$ a symmetric set 
such that for some positive integers $d,m$ there holds:
\begin{enumerate}
\item 
The order in $G$ of every element of $X$ is bigger than $2m$.
\item\label{item:2}
The equality $g^i=h^j$ with $g,h \in X$, $g\neq h^{\pm 1}$, and $-m\le i, j\le m$, holds only for $i=j=0$.
\item For every $g\in X$ the group element $g^m$ can be expressed as the 
product of at most $d$ elements of $X$.\label{item3}
\end{enumerate}
Then there is some $S^{-1}=S\subset G$ and a $G$-invariant consistent path system $\mathcal{P}$
in the Cayley Graph $\Gamma(G,S)$ for which $\Delta(\mathcal{P}) \ge \frac{m}{d|X|}$.
\end{lemma}

\begin{proof}
The set $S^{-1}=S$ is defined as follows:
\[S = X \cup \left(G\setminus \bigcup_{g\in X} \{g^{i} : -m\le i \le m\}\right).\]
In words, the set $S$ consists of all the elements of $X$ as well as every 
element of $G$ that cannot be expressed as a power
smaller than $m$ of some $g\in X$.\\
To construct the path system we follow the pattern suggested by \cref{prop:words} and
specify a collection of words $\{w_g: g\in G\}$ that satisfy
the consistency condition of that proposition. 
If $y= g^{\ell}$ for $g\in X$ and $0\le\ell \le m$ we set 
$w_y = g,\dots,g$, where $g$ appears $l$ times. 
Otherwise, if $y\in G$ cannot be expressed as a small power of some $g\in X$ we set 
$w_y=y$. Notice that by conditions 1 and 2 of the present lemma,
the words $\{w_y:y\in G\}$ are well defined 
and satisfy the consistency condition of \cref{prop:words}. 
These words then generate a $G$-invariant consistent path system where 
\[P_{e,y} = (e,g,\dots, g^{\ell}),\]
whenever $y= g^{\ell}$ for some $g\in X$ and $-m\le \ell \le m$.

We turn to derive a lower bound on $\Delta(\mathcal{P})$. 
Let $\alpha >0$ and $\rho$ a metric on $G$ such that 
\[\rho(P_{u,v}) \le \alpha\cdot \rho(u,v)\]
for all $u,v\in G$, $P_{u,v}\in \mathcal{P}$. By \cref{prop:invariant_metric} we may assume $\rho$ is $G$-invariant. 
Consider the paths $P_{e,g^m}$ over $g\in X$. 
Due to the $G$-invariance of $\rho$, we have 
\begin{equation}\label{eq:colinear}
\rho(P_{e,g^m}) = 
\sum_{j=1}^{m}\rho(P_{g^{j-1},g^j}) = m\cdot \rho(e,g).
\end{equation}
On the other hand, by assumption \ref{item3} of \cref{lem:cayley_metric},
we can write $g^m = h_1\cdots h_s$ where $h_1,\ldots,h_s$ belong to $X$ and $s\le d$.
In particular, the triangle inequality along the path 
$(e,h_1,h_1h_2,\dots,h_1\cdots h_s)$ yields
\begin{equation}\label{eq:triangel_ineq}
\rho(e,g^m)  \le \sum_{j=1}^{s}\rho(h_1\cdots h_{j-1},h_1\cdots h_{j}) =\sum_{j=1}^{s}\rho(e,h_{j}) \le d\cdot \max_{h\in X} \rho(e,h)
\end{equation}
Putting together \eqref{eq:colinear} and \eqref{eq:triangel_ineq} 
we get that for all $g\in X$
\[ m\cdot \rho(e,g) = 
\rho(P_{e,g^m}) \le \alpha\cdot \rho(e,g^m)\le 
\alpha d\cdot\max_{h\in X} \rho(e,h).\]
Summing over all $g\in X$ yields 
\[ m \sum_{g\in X}\rho(e,g) \le \alpha d|X|\cdot \max_{h\in G}\rho(e,h).\]
Consequently \[\frac{m}{d|X|} \le \alpha,\]
as claimed.
\end{proof}
We remark that if we restrict ourselves to the generating set $X$ then the lemma still holds for the {\em partial} consistent path system $\title{\mathcal{P}} = \{(e,g,\dots, g^\ell) : g\in X, \ell \le m\}$ in $\Gamma(G,X)$. Extending this set of generators allows us to extend this partial system into a full system on some larger Cayley graph.

As we show next, the cyclic group $Z_n$ of prime order $n$ 
has a subset $X$ that satisfies the conditions of 
\cref{lem:cayley_metric}. Since this group is abelian,
we use additive rather than multiplicative notation.
\begin{lemma}\label[lemma]{lem:cyclic}
Let $n$ be prime and $m:=\frac{\sqrt n}{\log^2 n}$. 
There exists $X\subset Z_n$, with $|X| \le O(\log n)$ such that
\begin{enumerate}
\item\label{short:gen} 
Every element in $Z_n$ can be expressed as the sum of $O(\log n)$ elements of $X$.
\item\label{no:repeats}
If $ia = jb$ for some $a\neq b\in X$ and $-m\le i,j \le m$, then $i=j=0$.
\end{enumerate}
\end{lemma}
Note that \cref{thm:main} directly follows from \cref{lem:cayley_metric} 
and \cref{lem:cyclic}. \\

To prove \cref{lem:cyclic} we use the following lemma from \cite{AR}.
\begin{lemma}\label{lem:expander}
There exists $k=O(\log n)$ such that if $G$ is a group of order $n$ and $X\subset G$ is a random subset of
$k$ uniformly independently chosen elements then with high probability $X$ generates $G$ with diameter 
at most $k$.
\end{lemma}
\begin{proof} (\Cref{lem:cyclic}) 
As in \cref{lem:expander}, let $G = Z_n$, and let $X\subset Z_n$ be
a random subset of $k$ uniformly independently chosen elements of $Z_n$,
where $k=O(\log n)$. Condition \ref{short:gen} of 
\cref{lem:cyclic} is satisfied with high probability by 
lemma \ref{lem:expander}.\\
We turn to condition \ref{no:repeats}. Let $a\neq b\in Z_n$
be a pair of elements that fails it, namely
\begin{equation}\label{collision}
ia=jb \text{~for some~} i, j \text{~with~} 1\le |i|,|j|\le m.
\end{equation}
We call $\{a,b\}$ {\em a forbidden pair}, and wish to show that
with almost certainty a randomly chosen $X$ 
contains no forbidden pair. Relation (\ref{collision}) 
clearly implies that $b = j^{-1} i a$. 
Consequently, any $a\in Z_n$ forms a forbidden pair $\{a,b\}$ with at most
$2m^2$ elements $b$ in $Z_n$. 
\begin{comment}    
This implies the number of $b\in Z_n$ such that $a,b$ is bad is bounded by 
the number of choices for $i$ and $j$, i.e. $2m^2$. Therefore, for fixed 
$1\le s<t\le k$
\[\text{Pr}(a_s,a_t \text{ is bad})\le \frac{4m^2}{n} =O(n^{-2\varepsilon}),\]
since $m=n^{\frac{1}{2}-\varepsilon}$. Applying the union bound yields,
\[\text{Pr}(a_s,a_t \text{ is bad for some } 1\le s<t\le k)\le \binom{k}{2}\frac{4m^2}{n} =O(n^{-2\varepsilon}\log^2n ).\]
\end{comment}
Therefore, the expected number of forbidden pairs in a
randomly chosen $X$ is at most
\[\binom{|X|}{2}\frac{2m^2}{n}\le O\left(\frac{1}{\log^2 n}\right)\]
By Markov's Inequality, with high probability $X$ 
contains no forbidden pairs.
 \end{proof}

\section{Realizable Values of $\Delta(\mathcal{P})$} \label{sec:realizable}
Alex Scott has asked us whether $\Delta(\mathcal{P})$ can take on irrational values. 
Here is our answer:
\begin{theorem}\label[theorem]{thm:algebraic}
For every $n$-point consistent path system $\mathcal{P}$, 
the real number $\Delta(\mathcal{P})$
is algebraic and (its minimal polynomial) can be determined in $\text{poly}(n)$ time.
\end{theorem}
and in particular,
\begin{theorem}\label[theorem]{thm:irrational}
There exist consistent path systems $\mathcal{P}$ for which $\Delta(\mathcal{P})$ is irrational.
\end{theorem}

Our proof of \cref{thm:algebraic} uses several notions and results from 
algorithmic algebra, starting with bit sizes. The bit-size of
an integer $r$ is defined to be 
$\text{size}(r) \coloneqq 1+\lceil\log(|r|+1)\rceil$. 
The bit-size of a reduced rational 
number $p/q$ is 
$\text{size}(p/q)\coloneqq \text{size}(p)+\text{size}(q) + \log(\text{size}(p))$. 
For a rational polynomial 
$F(X) \coloneqq \sum_{i=0}^{m}a_i X^i \in \mathbb{Q}[X]$ it is
$\text{size}(F)\coloneqq \sum_{i=0}^{m} \text{size}(a_i)+ \log(\text{size}(a_i))$. \\
The minimal polynomial of $\alpha \in \overline{\mathbb{Q}}$, an algebraic number
is denoted by $F_{\alpha}(X)\in \mathbb{Z}[X]$,
and $\text{size}(\alpha)$ is taken to be $\text{size}(F_{\alpha})$,
as defined above.\\
The following result says that any algebraic number
$\alpha \in \overline{\mathbb{Q}}$ can be efficiently recovered from
any ``good enough'' rational approximation of $\alpha$.

\begin{theorem}(e.g., \cite{Yap} chapter 9.6)\label[theorem]{thm:approx}
Given a rational number $a$ and a positive integer $N$
there is at most one algebraic number $\alpha$ of size $\le N$
with $|a-\alpha|\le 2^{-5N^3}$. If such $\alpha$ exists,
then its minimal polynomial can be found in $\text{poly}(N)$ time.
\end{theorem}

% \begin{theorem}(e.g., \cite{Yap})\label[theorem]{thm:approx}
%     There is an algorithm taking input $(a,N)\in \mathbb{Q} \times \mathbb{N}$ and running in $\text{poly}(N)$ time which returns $F_{\alpha}(X)$ assuming the following conditions hold:
%     \begin{enumerate}
%         \item $|a-\alpha|\le 2^{-5N^3}$
%         \item $\text{size}(F_{\alpha}(X))\le N$
%     \end{enumerate}
% \end{theorem}

The following observation immediately follows from the definitions: 
\begin{observation}\label[observation]{obs:LP}
Let $\mathcal{P}$ be a consistent path system with vertex set $V$. 
Associate a formal variable $x_{a,b} = x_{b,a}$ with every pair $a, b\in V$. 
The path system $\mathcal{P}$ is $t$-metric if and only if the following systems of 
linear inequalities is feasible:
    \begin{equation}\label{eq:LP1}
    \begin{split}
        x_{a,b} & \le x_{a,c} + x_{c,a} \hspace{2cm} \forall a,b,c\in V\\
        x_{a,a} &= 0 \hspace{3.6cm}\forall a\in V\\
        x_{a,b} & \ge 1 \hspace{3.6cm}\forall a,b \in V, a\neq b \\
         x_{a_{0},a_{1}} + \cdots + x_{a_{k-1},a_{k}} & \le t x_{u,v} \hspace{3cm} \forall u,v \in V, \ P_{u,v} = (u=a_0,\dots , a_{k} = v)
    \end{split}
    \end{equation}
\end{observation}
We now prove \cref{thm:algebraic}.
\begin{proof}(\Cref{thm:algebraic})
Set $\alpha = \Delta(\mathcal{P})$. In order to apply \Cref{thm:approx},
and prove \cref{thm:algebraic}, we need to show that
\begin{enumerate}
\item
$\alpha$ can be efficiently approximated very closely by rationals.
\item 
$\alpha$ is algebraic, and
\item 
its bit size is not too big.
\end{enumerate}

{\bf 1.} We start with the first item: Clearly,
LP \eqref{eq:LP1} is feasible whenever $t>\alpha$ and infeasible for $t<\alpha$,
and $1\le \alpha \le n$. 
Therefore, we can approximate $\alpha$ by applying binary search over the interval $[1,n]$ 
checking each time the feasibility of \eqref{eq:LP1}. By classical LP theory, (e.g., \cite{BT}), 
the feasibility of \eqref{eq:LP1} can be decided in
time $\text{poly}(n,\text{size}(t))$. In particular, if $N=\text{poly}(n)$ then an approximation 
satisfying $|a-\alpha|\le 2^{-5N^3}$ can be found in time polynomial in $n$.

{\bf 2.} That $\alpha$ is algebraic follows from \cref{obs:LP} and fundamental results in 
real algebraic  geometry, e.g., chapter 2 in \cite{SPR}, but we turn to give a direct proof.
Let us write the system of linear inequalities in \eqref{eq:LP1} as $Ax\le b$ where 
$A = A(t)$ is a $m\times k$ matrix with entries in $\{0\pm 1, \pm t\}$ and  
$b\in \mathbb{R}^{m}$  with entries in $\{0,\pm 1\}$. Note that $m=\Theta(n^3)$ and 
$k=\Theta(n^2)$. 
By basic LP theory (e.g., \cite{BT}) if this system is feasible then it admits 
a basic feasible solution (BFS). 
In our case, that would be a non-negative vector $y\in \mathbb{R}^k$ that
is defined by $A_Iy = b_I$, where $I\subset [m]$ is a subset of $k$ row indices and $A_I$ is a non-singular $k\times k$ minor of $A$ obtained by restricting to the rows in $I$, similarly with $b_I$.\\
As mentioned, the system $A(t) x \le b$ is feasible for all $t >\alpha$. 
So, there must exist a subset $I\subset [m]$ such that the set 
$\{t>\alpha : \text{$y(t) = A_I(t)^{-1}b_I$ is a BFS}\}$ has $\alpha$ as a limit point.
Cramer's rule yields an explicit expression for the BFS $y(t)$.
Namely, for $1\le i \le k$, its $i$-th coordinate equals
$y_i(t) = \frac{\det(A_{I,i}(t))}{\det(A_{I,0}(t))}$,
where the matrix $A_{I,i}$ is obtained from $A_{I}$ by replacing 
$i$-th column in $A_{I}$ by $b_I$ and $A_{I,0}$ stands for $A_{I}$.
Let $f_i(t) := \det(A_{I,i}(t))$ and observe that $f_i$ is a polynomial in $t$
with $\text{size}(f_i)\le k^3 = O(n^6)$. The bound on the size follows, since
$A_{I,i}$ is a $k\times k$ matrix with entries in $\{0,\pm 1, \pm t\}$.\\ 
If $\alpha$ is a root of $f_0(t)$ then it is algebraic and we are done. 
Otherwise, by continuity, $y_i(\alpha) = \frac{f_i(\alpha)}{f_0(\alpha)}$ 
(with $i=1,\ldots,k$) is a solution to the system $A(\alpha)x\le b$. 
Now consider what happens for $t$ slightly smaller than $\alpha$. The expression
$y_i(t) = \frac{f_i(t)}{f_0(t)}$ is still well-defined, and yet $A(t)x\le b$ is 
infeasible, and at least one of these inequalities fails. Namely,
there exists $1\le j \le m$ such that the set 
$\{ t<\alpha: a_{j,1}y_1(t) + \cdots + a_{j,k}y_k(t) - b_j > 0\}$
has $\alpha$ as a limit point. On the other hand, since $y(\alpha)$ is feasible solution, it must be that $a_{j,1}y_1(\alpha) + \cdots + a_{j,k}y_k(\alpha) - b_j =0$. 
Since $y_i(t) = \frac{f_i(t)}{f_0(t)}$, we  conclude that the polynomial 
\[g(t) = a_{j,1}f_1(t) + \cdots + a_{j,k}f_k(t) - b_j f_0(t)\]
is non-constant polynomial with root $\alpha$. Moreover, 
$\text{size}(g)$ is also bounded by $k^3 = O(n^6)$.

{\bf 3.} We need to show that $\text{size}(\alpha)\le O(\text{poly}(n))$. 
We already know that $\alpha$ is a root of some integer polynomial whose bit size
we managed to bound. However,
we need to bound the size of $F_{\alpha}$, its minimal polynomial. The missing
ingredient is the Landau-Mignotte bound (\cite{Yap} chapter 4.5). It says 
that $\|\Phi\|_1\le 2^d \|\Psi\|_1$ for any two polynomials $\Phi(X), \Psi(X) \in \mathbb{Z}[X]$
such that $\Phi$ divides $\Psi$. Here $\|\cdot\|_1$ denotes the $\ell^1$ norm of the 
coefficients and $d=\text{deg}(\Psi)$.
\end{proof}

We now prove \cref{thm:irrational} by providing an explicit example of a path system
for which $\Delta$ is irrational.
We first recall some basic properties of {\em Paley graphs} \cite{Paley}.
Associated with every prime $p\equiv 1 \pmod 4$ 
is a graph $G_{p}$ with vertex set
$\mathbb{F}_{p}$, where the vertices 
$a$ and $b$ are neighbors in $G_{p}$ if and only $a-b$ is a quadratic residue
in $\mathbb{F}_{p}$. In other words, $G_{p}$ is the Cayley graph
of the cyclic group $Z_p$ whose generators are the quadratic residues of 
$\mathbb{F}_{p}$. We have constructed in \cite{CL2}
an infinite family of non-metrizable $Z_p$-invariant path systems over such graphs,
$\mathcal{P}_{p}$ that are defined as follows: for $a,b\in \mathbb{F}_{p}$
\begin{itemize}
    \item If $a-b$ is a quadratic residue then $P_{a,b} = (a,b)$
    \item If $a-b\not \equiv \pm 3 \pmod p$ is a quadratic non-residue, then $P_{a,b} =\left( a,\frac{a+b}{2}, b\right) $
    \item If $a-b\equiv \pm 3 \pmod p$ then $P_{a,b} =\left( a,a\pm 1,a\pm 2, b\right)$
\end{itemize}
Here, we consider the case $p=29$, and prove the following claim:
\begin{claim}
$\Delta(\mathcal{P}_{29})$ is the middle root of the polynomial $2x^3-3x^2-10x+12$. 
Its numerical value is $\approx 1.103$, and it is irrational.
\end{claim}

\begin{proof}
The irrationality of $\Delta(\mathcal{P}_{29})$ follows, since the polynomial
$2x^3-3x^2-10x+12$ is easily seen to be irreducible 
(e.g., it has no roots in $\mathbb{F}_5$). It is also easily verifiable that
this polynomial has three real roots $r_1<r_2<r_3$, and we set $r=r_2\approx 1.103$. 
To prove that $\Delta(\mathcal{P}_{29})\le r$, we define an explicit
group-invariant (closed under addition $\bmod 29$)
metric $\rho$ on $\mathbb{F}_{29}$ that approximates $\mathcal{P}_{29}$ 
to a factor of $r$. To this end we specify
$\rho(u,v) = w_a$ whenever $u-v\equiv\pm a\pmod{29}$ 
as follows:
        \begin{equation*}
	\begin{split}
		w_{1} & = r^2 \\
		w_{2} & = 2r \\
		w_{3} & = 3r \\
		w_{4} & = 3r  - r^2\\
		w_{5} & = 3r  - r^2
		\end{split}
    \hspace{2.5cm}
        \begin{split}
		w_{6} & = 6- 2r - r^2\\
		w_{7} & = 6 - 2r-r^2 \\
        w_{8} & = 6- 2r\\
		w_{9} & = 6 - 2r-r^2\\
		w_{10} & = 6 - 2r
		\end{split}
		\hspace{2.5cm}
    \begin{split}
		w_{11} & = 6 + r-2r^2\\
		w_{12} & = 6 + r-2r^2\\
		w_{13} & = 6 + r-3r^2\\
		w_{14} & = 6 + r-2r^2
	\end{split}
\end{equation*}
    It can be verified that $\rho$ defines a metric and that $\rho(P_{u,v})\le r \rho(u,v)$ for all $u,v\in \mathbb{F}_{29}$.\\
    We now argue that $\Delta(\mathcal{P}_{29})\ge r$. Recall from earlier that $\alpha$-metrizability can be reframed as an LP feasibility problem, \cref{obs:LP}.
In light of \cref{prop:invariant_metric}, if $\mathcal{P}$ is $G$-invariant we may assume that $x_{a,b}=x_{u,v}$ whenever $b-a \equiv v-u \bmod~29$.
In our specific case, the system $\mathcal{P}_p$ is $t$-metric iff the following system of inequalities is satisfied: 
\begin{equation}\label{eq:LP}
    \begin{split}
        3w_1 &\le t w_3 \\
        2w_{a/2} & \le t w_{a} \hspace{2cm} \forall a\in \mathbb{F}_{p} \text{ s.t. $a\neq \pm 3$ is a non-residue} \\
        w_{a} & \le w_{b} + w_{c} \hspace{1.4cm} \forall a,b,c \in \mathbb{F}_{p} \  \text{ s.t. } a = b+c \\
        w_{a} & = w_{-a} \hspace{2.15cm} \forall a\in \mathbb{F}_{p}\\
        w_{0} & = 0 \\
        w_{a} &\ge  1 \hspace{2.7cm} \forall a\in \mathbb{F}_{p}
    \end{split}
\end{equation}
In our case, for $p=29$ we consider the following subsystem of \eqref{eq:LP} 
\begin{equation}\label{eq:reduced_LP}
	\begin{split}
		3w_{1}& \le t w_{3}\\
		2w_{4}& \le t w_{8}\\
		2w_{5}& \le t w_{10}\\
		2w_{9}& \le t w_{11}\\
		2w_{7}& \le t w_{14}
	\end{split}
	\hspace{5cm}
	\begin{split}
		w_{3}& \le w_{1}+w_{4}\\
		w_{8}& \le w_{1}+w_{9}\\
		w_{10}& \le w_{1}+w_{9}\\
		w_{11}& \le w_{4}+w_{7}\\
		w_{14}& \le w_{5}+w_{9}
	\end{split}
\end{equation}
where we have identified $w_{a} = w_{-a}$. In order to obtain a lower bound for $\Delta(\mathcal{P}_{29})$, we determine the minimum $t\ge 1$ such that the system \eqref{eq:reduced_LP} is feasible. We have already seen that $\mathcal{P}_{29}$ is $r$-metric, whence this minimum value is in the interval $[1,r]$. To determine feasibility we apply Fourier–Motzkin elimination \cite{SW} to the variables 
$w_{3}$, $w_{8}$, $w_{10}$, $w_{11}$, $w_{14}$, $w_{7}$, $w_{4}$, $w_{5}$, $w_{1}$,  in that order. This process yields the following inequality
\begin{equation}\label{eq:ineq}
	\begin{split}
		0& \le  w_{9}(-2t^4-t^3 + 16t^2 + 8t - 24) = -w_{9}(t+2)(2t^3-3t^2-10t+12).
	\end{split}
\end{equation}
In other words, the system \eqref{eq:reduced_LP} is feasible only if \eqref{eq:ineq} is satisfied.
Since $w_{9}$ is strictly positive, for $1\le t \le r$ the inequality \eqref{eq:ineq} is satisfied only when $t= r$.
\end{proof}

\section{Open Questions}
\begin{comment}
The obvious open question that remains to be answered is how large 
$\Delta(\mathcal{P})$ can be for an $n$-point space.
While we touched upon upper bounds for $\Delta(\mathcal{P})$, our best upper bound is still linear in $n$.
\begin{open}
Is it true that
    \[\Delta(\mathcal{P}) = o(n)\]
for any consistent path system $\mathcal{P}$ on $n$ points?
\end{open}
While our main construction was defined over certain Cayley graphs of the cyclic group, \cref{lem:cayley_metric} is quite general. We ask if better lower bounds can be obtained with other groups. 
\begin{open}
    Do there exist groups $G$ and a $G$-invariant path systems $\mathcal{P}$ satisfying 
    \[\Delta(\mathcal{P}) = \omega(\sqrt{|G|})?\]
\end{open}
\end{comment}
Many questions in this domain remain open, starting with the obvious ones:
\begin{open}
Consider $\max\Delta(\mathcal{P}_n)$
over all consistent $n$-point path systems $\mathcal{P}_n$.
\begin{itemize}
\item 
Can you close the gap between the upper and lower bounds established here?
\item 
Can you find other classes of graphs $G$ and path systems with 
large $\Delta$? For other Cayley graphs? For less symmetric graphs?
\end{itemize}
\end{open}
A path system in a graph $G=(V,E)$ is called {\em neighborly}
if it uses every edge in $G$. In other words, $P_{x,y}=xy$
whenever $xy\in E$. The questions that we conider here have their
graph-theoretic counterparts. In particular,
\begin{open}
Consider the following parameter of a graph $G$:
\[\delta(G)=\max\{\Delta(\mathcal{P})|\mathcal{P} \text{~is a consistent neighborly path system on~}G\}.\]
\begin{itemize}
\item 
What is the computational complexity of this graph parameter?
\item 
How is it distributed among $n$-vertex graphs?
\end{itemize}
\end{open}
\noindent
The parameter $\Delta$ induces a gradation on consistent
path systems. For $1\le t\le O(\sqrt{n})$, let
\[\cal{C}\rm_{n,t}:= \{\text{consistent~}
n\text{-point path systems~} \mathcal{P} \text{~with~} \Delta(\mathcal{P})\le t\}\]
With an eye to our paper \cite{CL3}, we ask:
\begin{open}
How does the cardinality of the set $|\cal{C}\rm_{n,t}|$ grow with 
$t$ and $n$?
\end{open}
The graphic version of this notion is also of interest: Let
\[\cal{G}\rm_{n,t}:=\{n\text{-vertex graphs~}G,
\text{~such that~}\Delta(\mathcal{P})\le t \text{~for every consistent path
system~}\mathcal{P}\text{~on~}G\}\]

In view of our earlier work \cite{CL, CCL2} we ask:

\begin{open}
Does the class $\cal{G}\rm_{n,t}$ have any interesting 
structural properties? Specifically, is it minor-closed?
Or closed under topological minors?
\end{open}

\section*{Acknowledgments}
We would like to thank Alex Scott for his insightful questions and discussions. 

\printbibliography

\end{document}